\documentclass[reqno]{amsart}
\usepackage{graphicx} 
\usepackage[utf8]{inputenc}
\usepackage{amsmath}

\usepackage{tikz}

\usetikzlibrary{shapes.geometric}
\usepackage{amsthm}

\usepackage{amsfonts}
\usepackage{caption}
\usepackage{amssymb}
\usepackage{verbatim}
\usepackage{mathtools}
\usepackage{subcaption}
\usepackage{changepage}
\newtheorem{theorem}{Theorem}
\newtheorem{lemma}[theorem]{Lemma}

\newtheorem{corollary}[theorem]{Corollary}
\newtheorem{definition}[theorem]{Definition}
\numberwithin{equation}{section}

\title{Examining $H$-Closed Ducci Sequences on $\mathbb{Z}_m^n$}

\author{Mark L. Lewis }
\address{Department of Mathematical Sciences\\
Kent State University\\
Kent, OH 44242}
\email{lewis@math.kent.edu}
\author{Shannon M. Tefft}
\address{Department of Mathematical Sciences\\
Kent State University\\
Kent, OH 44242}
\email{stefft@kent.edu}


\date{February 2025}

\begin{document}
\begin{abstract}
Let $D$ be an endomorphism on $\mathbb{Z}_m^n$ so that
\[D(x_1, x_2, ..., x_n)=(x_1+x_2 \; \text{mod} \; m, x_2+x_3 \; \text{mod} \; m, ..., x_n+x_1 \; \text{mod} \; m).\]
We call the sequence $\{D^{\alpha}(\mathbf{u})\}_{\alpha=0}^{\infty}$ the Ducci sequence of $\mathbf{u} \in \mathbb{Z}_m^n$, which always enters a cycle.
Now let $H$ be an endomorphism on $\mathbb{Z}_m^n$ such that 
\[H(x_1, x_2, ..., x_n)=(x_2, x_3, ..., x_n, x_1).\]
In this paper, we will talk about a few cases when $\mathbf{u}$ and $H^{\beta}(\mathbf{u})$ have the same Ducci cycle for $\beta > 0$, as well as prove a few cases of $n,m$ where this is guaranteed for every $\mathbf{u} \in \mathbb{Z}_m^n$.
\end{abstract}
\maketitle
\section{Introduction}
\indent Let $D$ be an endomorphism on $\mathbb{Z}_m^n$ such that 
\[D(x_1, x_2, ..., x_n)=(x_1+x_2 \; \text{mod} \; m, x_2+x_3 \; \text{mod} \; m, ..., x_n+x_1 \; \text{mod} \; m).\]
We call $D$ the \textbf{Ducci function} and we call the sequence $\{D^{\alpha}(\mathbf{u})\}_{\alpha=0}^{\infty}$ the \textbf{Ducci sequence of} $\mathbf{u} \in \mathbb{Z}_m^n$, with \cite{Breuer1, Ehrlich, Glaser} being a few sources that call it this.

\indent To better illustrate a Ducci sequence, we can consider the Ducci sequence of $(0,1,4) \in \mathbb{Z}_6^3$: $(0,1,4), (1,5,4),(0,3,5), (3,2,5), (5,1,2), (0,3,1), (3,4,1), (1,5,4)$. As one can see, if we continue listing terms in the Ducci sequence of $(0,1,4)$, then it will cycle through 6 tuples: $(1,5,4),(0,3,5), (3,2,5), (5,1,2), (0,3,1), (3,4,1)$. We call these tuples the Ducci cycle of $(0,1,4)$, or more formally,

\begin{definition}
The \textbf{Ducci cycle of} $\mathbf{u}$ is  
\[\{\mathbf{v} \mid \exists \alpha \in \mathbb{Z}^+ \cup \{0\}, \beta \in \mathbb{Z}^+  \ni \mathbf{v}=D^{\alpha+\beta}(\mathbf{u})=D^{\alpha}(\mathbf{u})\}.\]
The \textbf{length of} $\mathbf{u}$, $\mathbf{Len(u)}$, is the smallest $\alpha$ satisfying the equation 
\[\mathbf{v}=D^{\alpha+\beta}(\mathbf{u})=D^{\alpha}(\mathbf{u})\]
 for some $v \in \mathbb{Z}_m^n$ \and the \textbf{period of} $\mathbf{u}$, $\mathbf{Per(u)}$, is the smallest $\beta$ that satisfies the equation. 
\end{definition}
From this, we can see that $\text{Len}(0,1,4)=1$ and $\text{Per}(0,1,4)=6$. 

\indent A particularly important Ducci sequence is what we call the \textbf{basic Ducci sequence}, which is first discussed by \cite{Ehrlich} on page 302 and also by \cite{Breuer1, Dular, Glaser}. The basic Ducci sequence of $\mathbb{Z}_m^n$ is the Ducci sequence of $(0,0,...,0,1) \in \mathbb{Z}_m^n$. We denote $L_m(n)=\text{Len}(0,0,...,0,1)$ and $P_m(n)=\text{Per}(0,0,...,0,1)$. These notations are first used in Definition 5 of \cite{Breuer1} and $P_m(n)$ is similar to the notations that \cite{Ehrlich, Glaser} use. The reason why this sequence is significant is because of Lemma 1 of \cite{Breuer1}, which says that for any $\mathbf{u} \in \mathbb{Z}_m^n$, $\text{Len}(\mathbf{u}) \leq L_m(n)$ and $\text{Per}(\mathbf{u})|P_m(n)$. The notation of $P_m(n)$ is also used to represent the maximum value of the period in $\mathbb{Z}_m^n$ on page 858 of \cite{Breuer2}.

\indent Define $K(\mathbb{Z}_m^n)$ to be the set of all tuples in $\mathbb{Z}_m^n$ that belong to a Ducci cycle for some $\mathbf{u} \in \mathbb{Z}_m^n$. This subgroup is first defined by \cite{Breuer1} in Definition 4, where they also note that $K(\mathbb{Z}_m^n)$ is a subgroup of $\mathbb{Z}_m^n$. A proof of this is provided in Theorem 1 of \cite{Paper1}.
 
 \indent We now introduce another endomorphism on $\mathbb{Z}_m^n$, call it $H$, and define it so 
 \[H(x_1, x_2, ..., x_n)=(x_2, x_3, ..., x_n, x_1).\]
 This endomorphism is first defined on page 302 of \cite{Ehrlich}, as well as by \cite{Breuer1, Glaser, Paper1}. It also satisfies the condition that $D=I+H$ where $I$ is the identity endomorphism. In this paper, we are most interested in how $D$ and $H$ interact with each other. For example, in some cases, we will have that for $\mathbf{u} \in \mathbb{Z}_m^n$, and every $\mathbf{v}$ in the Ducci cycle of $\mathbf{u}$, $H^{\beta}(\mathbf{v})$ also belongs to the Ducci cycle of $\mathbf{u}$ for every $-n < \beta <n$. If a Ducci sequence exhibits this behavior, then we call the sequence $\mathbf{H}$-\textbf{closed}. If every Ducci sequence on $\mathbb{Z}_m^n$ exhibits this behavior, then we call $\mathbb{Z}_m^n$ $H$-closed. 
 
 \indent We would like to prove a few cases where we have $\mathbb{Z}_m^n$ is $H$-closed. We begin by narrowing to the case where $n=3$ and aim to prove the following theorem:
 \begin{theorem}\label{HclosedTheorem}
$\mathbb{Z}_m^3$ is H-closed if one of the following occurs:
\begin{enumerate}
    \item $m$ is a power of $2$. Specifically, if $\mathbf{u} \in K(\mathbb{Z}_m^3)$, then $D^2(\mathbf{u})=H(\mathbf{u})$ for every $\mathbb{Z}_m^3$.
    \item $m$ is prime and $m \equiv 5 \; \text{mod} \; 6$. Specifically, $D^{m-1}(\mathbf{u})=H^2(\mathbf{u})$ for every $\mathbf{u} \in \mathbb{Z}_m^3$.
    \item $m=2^lp$ for some $l \in \mathbb{Z}^+$ and $p \equiv 5 \; \text{mod} \; 6$ prime. Specifically, if $\mathbf{u} \in K(\mathbb{Z}_m^3)$, then $D^{p-1}(\mathbf{u})=\mathbf{u}$.
\end{enumerate}
\end{theorem}

\indent We will then examine some cases when $n \neq 3$ and observe some patterns in when $\mathbb{Z}_m^n$ is $H$-closed. We would like to prove one particular case of this. Here, $n$ is even, and we would like to specifically prove:
\begin{theorem}\label{Hclosed_even}
Let $n$ be even and assume $m \equiv -1 \; \text{mod} \; n$ is prime. Then $\mathbb{Z}_m^n$ is $H$-closed. Namely, we have that if $\mathbf{u} \in K(\mathbb{Z}_m^n)$, then
\[D^{m-1}(\mathbf{u})=H^{-1}(\mathbf{u}).\]
\end{theorem}

\indent The work in this paper was done while the second author was a Ph.D. student at Kent State University under the advisement of the first author and will appear as part of the second author's dissertation. 

 
 \section{Background}
 \indent Our definition of $D$ comes from a generalization of a Ducci function, $\bar{D}: \mathbb{Z}^n \to \mathbb{Z}^n$ or $\bar{D}: (\mathbb{Z}^+ \cup \{0\})^n \to (\mathbb{Z}^+ \cup \{0\})^n$, where 
 \[\bar{D}(x_1, x_2,..., x_n)=(|x_1-x_2|, |x_2-x_3|, ..., |x_n-x_1|).\]
 This version of Ducci is examined in \cite{Ehrlich, Freedman, Glaser, Furno}. Note that if $\mathbf{u} \in \mathbb{Z}^n$, then $\bar{D}(\mathbf{u}) \in (\mathbb{Z}^+ \cup \{0\})^n$. For this reason, we will refer to both the case of Ducci on $\mathbb{Z}^n$ and $(\mathbb{Z}^+ \cup \{0\})^n$ as the Ducci case on $\mathbb{Z}^n$.  There are also some sources, for example, \cite{Brown, Chamberland, Schinzel}, that use this same definition for $\bar{D}$ but on $\mathbb{R}^n$. 
 
 \indent There are some findings for Ducci on $\mathbb{Z}^n$ that are worth noting. First, every Ducci sequence still enters a cycle, as  \cite{Burmester, Ehrlich, Glaser, Furno} all discuss. In Lemma 3 of \cite{Furno}, it is proved that by the time the Ducci sequence enters its cycle, all of the entries of the sequence belong to the set $\{0,c\}$ where $c \in \mathbb{Z}^+$. This means that the Ducci case on $\mathbb{Z}_m^n$ that we defined at the beginning of this paper is important when $m=2$ for the Ducci case on $\mathbb{Z}^n$. 
 
 
 \indent The first paper to look at Ducci defined on $\mathbb{Z}_m^n$ is \cite{Wong}, and it is also looked at in \cite{Breuer1, Breuer2, Dular}. 
 
 \indent We now examine an example of a Ducci sequence more closely. We consider a transition graph that maps out all of the Ducci sequences of $\mathbb{Z}_6^3$ and look at the connected component containing the basic Ducci sequence of $\mathbb{Z}_6^3$, provided in Figure \ref{transgraph_notHclosed}.
 
 \begin{figure}
\centering
\begin{adjustwidth}{-50 pt}{-50 pt}

\begin{tikzpicture}[node distance={30mm}, thick, main/.style = {draw, circle}]
\node[main](1){$(0,1,1)$};
\node[main](2)[above left of=1]{$(0,0,1)$};
\node[main](3)[right of=1]{$(1,2,1)$};
\node[main](4)[above right of=3]{$(3,4,4)$};
\node[main](5)[below right of=3]{$(3,3,2)$};
\node[main](6)[right of=5]{$(4,5,4)$};
\node[main](7)[below left of=5]{$(0,5,5)$};
\node[main](8)[below right of=7]{$(0,0,5)$};
\node[main](9)[left of=7]{$(5,4,5)$};
\node[main](10)[below left of=9]{$(3,2,2)$};
\node[main](11)[above left of=9]{$(3,3,4)$};
\node[main](12)[left of=11]{$(2,1,2)$};

\draw[->](2)--(1);
\draw[->](1)--(3);
\draw[->](4)--(3);
\draw[->](3)--(5);
\draw[->](6)--(5);
\draw[->](5)--(7);
\draw[->](8)--(7);
\draw[->](7)--(9);
\draw[->](10)--(9);
\draw[->](9)--(11);
\draw[->](12)--(11);
\draw[->](11)--(1);

\end{tikzpicture}
\end{adjustwidth}
\caption{Transition Graph for $\mathbb{Z}_6^3$}
\label{transgraph_notHclosed}
\end{figure}
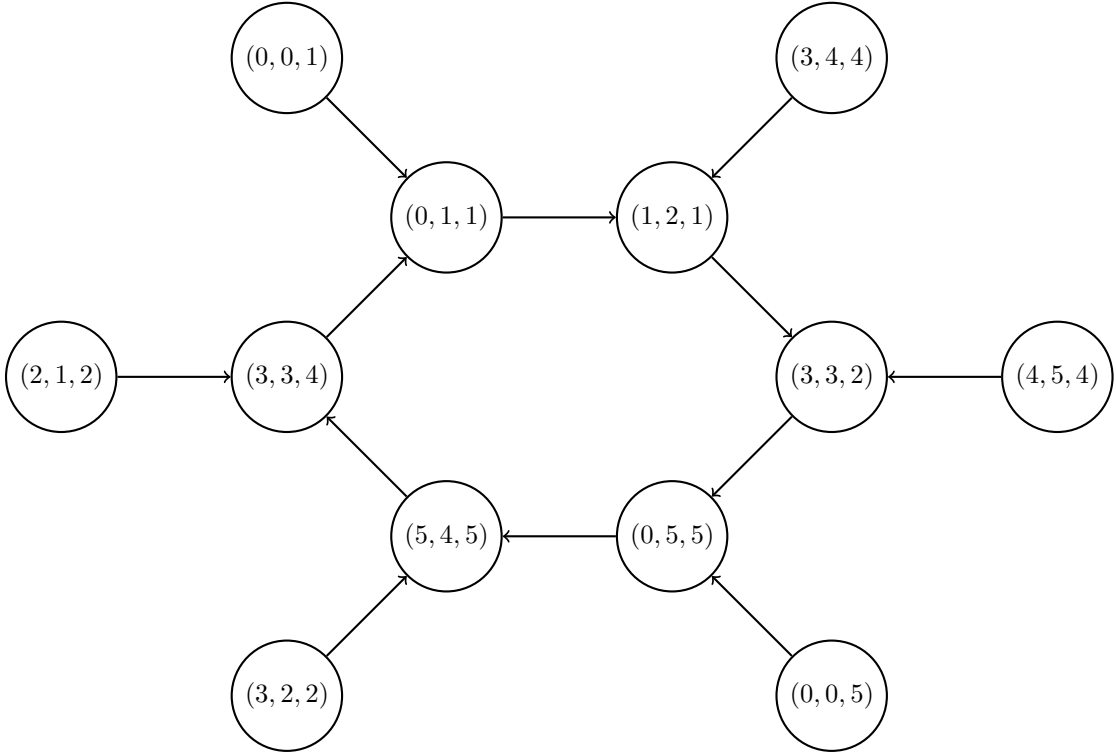
 
 \indent Notice that this sequence is not $H$-closed, which also means that $\mathbb{Z}_6^3$ is not $H$-closed. This begs the question: if $\mathbf{u}$ appears in Figure \ref{transgraph_notHclosed}, then what does the Ducci sequence of $H(\mathbf{u})$ look like? We first note that $H,D$ commute, which is proved on page 5 of \cite{Paper1} and is stated in \cite{Breuer1, Ehrlich, Glaser}. This means that for any $\mathbf{u} \in \mathbb{Z}_m^n$ and $-n < \beta <n$, the Ducci sequence of $H^{\beta}(\mathbf{u})$ is $\{H^{\beta}(D^{\alpha}(\mathbf{u}))\}_{\alpha=0}^{\infty}$. It also means that if $\mathbf{u} \in K(\mathbb{Z}_m^n)$, then $H^{\beta}(\mathbf{u}) \in K(\mathbb{Z}_m^n)$ for $-n < \beta <n$. 
 
 \indent Now if $\mathbf{u} \in \mathbb{Z}_m^n$, let $S(\mathbf{u})$ be the set of all tuples contained in the connected component of the transition graph of $\mathbb{Z}_m^n$ containing $\mathbf{u}$. Note this includes all tuples in the Ducci sequence of $\mathbf{u}$, as well as all tuples in $\mathbb{Z}_m^n$ who have the same Ducci cycle as $\mathbf{u}$. Suppose $\mathbf{u} \in \mathbb{Z}_m^n$, $-n < \beta <n$, $\beta \neq 0$, and $H^{\beta}(\mathbf{u}) \not \in S(\mathbf{u})$. Then $S(H^{\beta}(\mathbf{u}))=\{H^{\beta}(\mathbf{z}) \; | \; \mathbf{z} \in S(\mathbf{u})\}$ and the connected components of the transition graph containing $\mathbf{u}$ and $H^{\beta}(\mathbf{u})$ look the same, with the exception that if $\mathbf{z} \in S(\mathbf{u})$, then in the connected component of $H^{\beta}(\mathbf{u})$, $\mathbf{z}$ is replaced with $H^{\beta}(\mathbf{z})$. 
 
 \indent Note that by definition, the Ducci sequence of $\mathbf{u}$ is $H$-closed if and only if $S(\mathbf{u})=S(H^{\beta}(\mathbf{u}))$ for every $-n < \beta <n$.
 
 \indent This leads us to the four possibilities of what can happen with the Ducci sequences of $\mathbf{u}$ and $H^{\beta}(\mathbf{u})$ for $\mathbf{u} \in \mathbb{Z}_m^n$ and $-n < \beta <n$. The first case is the one depicted in Figure \ref{transgraph_notHclosed}: $\mathbf{u}$ satisfies the condition that $S(\mathbf{u}) \cap S(H^{\beta}(\mathbf{u}))$ is empty for every $-n < \beta < n$, $\beta \neq 0$. The second possibility is when the Ducci sequence is $H$-closed. In some cases, $\mathbb{Z}_m^n$ contains both $H$-closed sequences and sequences that are not $H$-closed. For example, in $\mathbb{Z}_6^3$, we have seen that the basic Ducci sequence is not $H$-closed. However, the Ducci sequence of $(1,2,3) \in \mathbb{Z}_6^3$ is $H$-closed because $D^2(1,2,3)=(2,3,1)$. In other cases, every Ducci sequence in $\mathbb{Z}_m^n$ is $H$-closed. An example is $\mathbb{Z}_{10}^3$; the basic Ducci sequence of which is provided in Figure \ref{transgraph_Hclosed}.
 
 \begin{figure}
 \centering
 \resizebox{.95\textwidth}{!}{
 \begin{tikzpicture}[node distance={30mm}, thick, main/.style = {draw, circle}]
 
\node[main](1){$(0,1,1)$};
\node[main](2)[right of=1] {$(1,2,1)$};
\node[main](3)[right of=2]{$(3,3,2)$};
\node[main](4)[right of=3]{$(6,5,5)$};
\node[main](5)[below of=4]{$(1,0,1)$};
\node[main](6)[below of=5]{$(1,1,2)$};
\node[main](7)[below of=6]{$(2,3,3)$};
\node[main](8)[left of=7]{$(5,6,5)$};
\node[main](9)[left of=8]{$(1,1,0)$};
\node[main](10)[left of=9]{$(2,1,1)$};
\node[main](11)[above of=10]{$(3,2,3)$};
\node[main](12)[above of=11]{$(3,2,3)$};

\node[main](13)[above left of=1]{$(0,0,1)$};
\node[main](14)[above of=2]{$(5,6,6)$};
\node[main](15)[above of=3]{$(6,7,6)$};
\node[main](16)[above right of=4]{$(8,8,7)$};
\node[main](17)[right of=5]{$(1,0,0)$};
\node[main](18)[right of=6]{$(6,5,6)$};
\node[main](19)[below right of=7]{$(6,6,7)$};
\node[main](20)[below of=8]{$(7,8,8)$};
\node[main](21)[below of=9]{$(0,1,0)$};
\node[main](22)[below left of=10]{$(6,6,5)$};
\node[main](23)[left of=11]{$(7,6,6)$};
\node[main](24)[left of=12]{$(0,0,1)$};

\draw[->](1)--(2);
\draw[->](2)--(3);
\draw[->](3)--(4);
\draw[->](4)--(5);
\draw[->](5)--(6);
\draw[->](6)--(7);
\draw[->](7)--(8);
\draw[->](8)--(9);
\draw[->](9)--(10);
\draw[->](10)--(11);
\draw[->](11)--(12);
\draw[->](12)--(1);

\draw[->](13)--(1);
\draw[->](14)--(2);
\draw[->](15)--(3);
\draw[->](16)--(4);
\draw[->](17)--(5);
\draw[->](18)--(6);
\draw[->](19)--(7);
\draw[->](20)--(8);
\draw[->](21)--(9);
\draw[->](22)--(10);
\draw[->](23)--(11);
\draw[->](24)--(12);
 
 \end{tikzpicture}
 }
 \caption{Transition Graph for $\mathbb{Z}_{10}^3$}
 \label{transgraph_Hclosed}
 \end{figure}
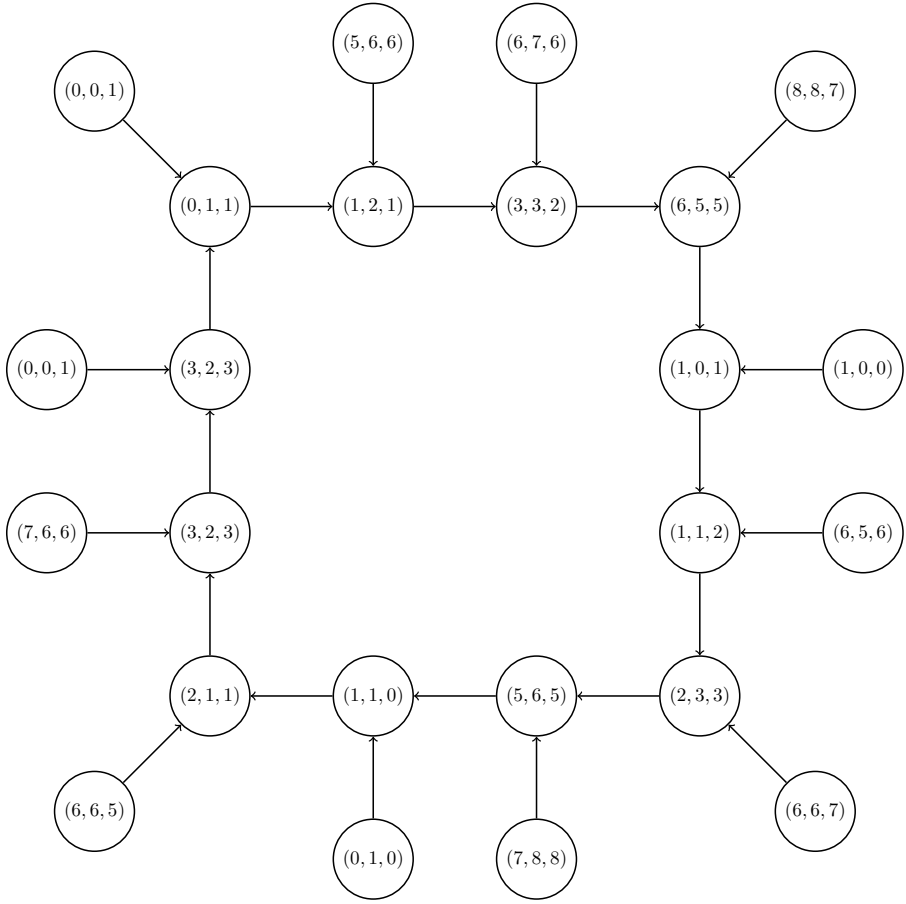
 
 \indent For the third possibility, we first need to note that for certain $n,m$, the Ducci cycle of all Ducci sequences in $\mathbb{Z}_m^n$ is $(0,0,...,0,0)$. If $n$ and $m$ are both powers of $2$, this is always the case. This is first proven in (I) on page 103 of \cite{Wong} and again by \cite{Dular, Paper1}. For this reason, these cases are $H$-closed because every tuple in $\mathbb{Z}_m^n$ has the same Ducci cycle, but do not have an $\alpha, \beta$ such that $D^{\alpha}(\mathbf{v})=H^{\beta}(\mathbf{v})$ for every $\mathbf{v} \in K(\mathbb{Z}_m^n)$ that we are intersted in because $K(\mathbb{Z}_m^n)=\{(0,0,...,0)\}$. 
 However, for every other case where the Ducci sequence of $\mathbf{u} \in \mathbb{Z}_m^n$ is $H$-closed and $\mathbf{v} \in K(\mathbb{Z}_m^n)$, there is $\alpha>0$ and $-n < \beta < n$, $\beta \neq 0$ such that $D^{\alpha}(\mathbf{v})=H^{\beta}(\mathbf{v})$.
 
 \indent The last possibility is when $\mathbf{u}$ satisfies the condition that $S(\mathbf{u}) \cap S(H(\mathbf{u}))$ is empty but there exists $-n < \beta < n$, $\beta \neq -1, 0, 1$ and $gcf(\beta, n) \neq 1$ such that $S(\mathbf{u}) =S(H^{\beta}(\mathbf{u}))$. Sometimes there are individual Ducci sequences in $\mathbb{Z}_m^n$ that satisfy this and for some $n,m$, all $\mathbf{u} \in \mathbb{Z}_m^n$ satisfy this condition. For example, if $\mathbf{v} \in K(\mathbb{Z}_3^{12})$, then $D^6(\mathbf{v})=H^{-3}(\mathbf{v})$ and for every $\mathbf{u} \in \mathbb{Z}_3^{12}$, $S(\mathbf{u})=S(H^{\beta}(\mathbf{u}))$ when $\beta \in \{0,3,6,9\}$, which we determined using MATLAB. We will talk more about how we used MATLAB to determine this in Section \ref{Hclosedsec_othercases}. Information on MATLAB can be found at \cite{MATLAB}. For this case, define the Ducci sequence of $\mathbf{u} \in \mathbb{Z}_m^n$ to be \textbf{weakly} $\mathbf{H}$-\textbf{closed} if there exists $-n < \beta <n$, $\beta \neq 0$ such that $S(\mathbf{u})=S(H^{\beta}(\mathbf{u}))$. Call $\mathbb{Z}_m^n$ weakly $H$-closed if all Ducci sequences in $\mathbb{Z}_m^n$ are weakly $H$-closed. Note by how it is defined, all $H$-closed sequences are weakly $H$-closed, but not all weakly $H$-closed sequences are $H$-closed. 
 
 \indent Now, for our proofs of Theorems \ref{HclosedTheorem} and \ref{Hclosed_even}, we will need a way to prove that $\mathbb{Z}_m^n$ is $H$-closed. 
 We note that if $L=L_m(n)$, then $D^L(\mathbf{u}) \in K(\mathbb{Z}_m^n)$ for every $\mathbf{u} \in \mathbb{Z}_m^n$ by Lemma 1 in \cite{Breuer1}. Also, because $H,D$ commute, if $\mathbf{v} \in K(\mathbb{Z}_m^n)$, $\gamma \geq 0$, and if there exists $\alpha, \beta$ such that 
 \[D^{\alpha}(\mathbf{v})=H^{\beta}(\mathbf{v}),\]
 then 
 \[D^{\gamma+\alpha}(\mathbf{v})=H^{\beta}(D^{\gamma}(\mathbf{u})).\]
 With these pieces in mind, we can prove the following lemma:
 
 \begin{lemma}\label{Hclosed_basic}
 $\mathbb{Z}_m^n$ is $H$-closed if and only if the basic Ducci sequence of $\mathbb{Z}_m^n$ is $H$-closed. 
 \end{lemma}
 \begin{proof}
 $(\Rightarrow)$ If $\mathbb{Z}_m^n$ is $H$-closed, then all Ducci sequences in $\mathbb{Z}_m^n$ are $H$-closed, including the basic Ducci sequence.
 
 $(\Leftarrow)$ Assume the basic Ducci sequence of $\mathbb{Z}_m^n$ is $H$-closed. If $n,m$ are both powers of $2$, then $\mathbb{Z}_m^n$ is $H$-closed, so we may assume that $n,m$ are not both powers of $2$. Let $L=L_m(n)$. Then there exists $\alpha>0$ and $-n < \beta < n$, $\beta \neq 0$ such that 
 \begin{equation}\label{Eq_Hclosed_basic}
 D^{L+\alpha}(0,0,...,0,1)=H^{\beta}(D^L(0,0,...,0,1)).
 \end{equation}
 It suffices to show 
 \[D^{L+\alpha}(\mathbf{u})=H^{\beta}(D^L(\mathbf{u}))\]
 for every $\mathbf{u} \in \mathbb{Z}_m^n$.
 
\indent We now note that every $\mathbf{u}=(x_1, x_2, ..., x_n) \in \mathbb{Z}_m^n$ can be broken down as
 \[\mathbf{u} =\sum_{s=1}^n x_sH^{-s}(0,0,...,0,1).\]
 This specific breakdown is used on page 6 of \cite{Paper1}, and before that, a more general breakdown for Ducci sequences on Abelian groups is provided in the proof of Lemma 1 of \cite{Breuer1}. Applying Ducci to both sides yields
 \[D^{L+\alpha}(\mathbf{u})=\sum_{s=1}^nx_sH^{-s}(D^{L+\alpha}(0,0,...,0,1)\]
 because $H,D$ commute and because for $\lambda \in \mathbb{Z}_m$ and $\mathbf{v} \in \mathbb{Z}_m^n$, $D(\lambda \mathbf{v})=\lambda D(\mathbf{v})$ like in the original Ducci case discussed earlier. Using Equation (\ref{Eq_Hclosed_basic}), this is
 \[D^{L+\alpha}(\mathbf{u})=\sum_{s=1}^n x_s H^{-s}(H^{\beta}(D^L(0,0,...,0,1))).\]
 Rearranging, this is 
 \[D^{L+\alpha}(\mathbf{u})=H^{\beta}(D^L(\sum_{s=1}^n x_s H^{-s}(0,0,...,0,1))),\]
 which is $H^{\beta}(D^L(\mathbf{u}))$ and the lemma follows.

 \end{proof}
 
 \indent This lemma will allow us to prove $\mathbb{Z}_m^n$ is $H$-closed by proving its basic Ducci sequence is $H$-closed. 
 
 \indent We introduce one more tool that we will use to prove Theorems \ref{HclosedTheorem} and \ref{Hclosed_even}. To do this, we examine the first few entries of the Ducci sequence of $(x_1, x_2, ...., x_n)$.
 
 \[(x_1, x_2, ..., x_n)\]
     \[(x_1+x_2, x_2+x_3, ..., x_n+x_1)\]
     \[(x_1+2x_2+x_3, x_2+2x_3+x_4, ..., x_n+2x_1+x_2\]
     \[(x_1+3x_2+3x_3+x_4, x_2+3x_3+3x_4+x_5, ..., x_n+3x_1+3x_2+x_3)\]
     \[\vdots\]  
     
     Notice that for each tuple in the Ducci sequence, the coefficient on $x_1$ is the same as the one on $x_2$ in the second entry, the same as the one on $x_3$ in the third entry, and so on. We take advantage of this pattern by defining $a_{r,s}$ to be the coefficient on $x_{s-i+1}$ in the $i$th entry of $D^r(x_1, x_2, ..., x_n)$ where $r>0$ and $1 \leq s \leq n$. We reduce the $s$ coordinate in $a_{r,s}$ modulo $n$, where we say that coordinate is $n$ when $s$ is divisible by $n$. A more thorough introduction of these $a_{r,s}$ coefficients is provided on page 6 of \cite{Paper1}. 
     
     \indent There are a few findings about $a_{r,s}$ that will prove useful. The first comes from Theorem 5 of \cite{Paper1}, which allows us to break $a_{r,s}$ down and identify the value of $a_{r,s}$ when $r$ is small. Let $r>t>0$ and $1 \leq s \leq n$. Then 
     \begin{itemize}
     \item $a_{r,s}=a_{r-1,s}+a_{r-1, s-1}$.
     
     \item $a_{r,s}=\displaystyle{\sum_{i=1}^n a_{t,i}a_{r-t,s-i+1}}$.
     
     \item Let $r<n$, then $a_{r,s}=\displaystyle{\binom{r}{s-1}}$.
\end{itemize}      
 
 
 \indent 

 In addition to the sum formula in the last bullet point, there is another sum breakdown. Originally found on page 103 of \cite{Wong} and officially stated in Proposition 2.1 of \cite{Dular}, this breakdown says the $i$th entry of $D^r(x_1, x_, ..., x_n)$ is equivalent to $\displaystyle{\sum_{j=0}^r \binom{r}{j}x_{i+j} \; \text{mod} \; m}$. We can adapt this to be specifically in terms of $a_{r,s}$, yielding the following lemma:
 
 \begin{lemma}\label{2nd_ars_sum}
 Let $r>t>0$, and $1 \leq s \leq n$. Then 
 \[a_{r,s}=\sum_{i=0}^t \binom{t}{i} a_{r-t,s-i}.\]
 \end{lemma}
 \begin{proof}
 We prove this via induction, with the $t=1$ case serving as the basis case.
 
 \indent \textbf{Inductive Step:} Assume it is true for $t-1$, or in other words, assume
 \[a_{r,s}=\sum_{i=0}^{t-1} \binom{t-1}{i} a_{r-t+1,s-i}.\]
 Then, by Theorem 5 of \cite{Paper1}, this is
 \[\sum_{i=0}^{t-1} \binom{t-1}{i}a_{r-t,s-i}+\sum_{i=0}^{t-1}\binom{t-1}{i} a_{r-t, s-i-1}.\]
 Changing the indices on the second sum, this is 
 \[\sum_{i=0}^{t-1} \binom{t-1}{i}a_{r-t,s-i}+\sum_{i=1}^t \binom{t-1}{i-1} a_{r-t, s-i}.\]
 Next, we pull out the $i=0$ term from the first sum and the $i=t$ term from the second sum to obtain
 \[\binom{t-1}{0}a_{r-t,s}+\binom{t-1}{t-1}a_{r-t,s-t}+\sum_{i=1}^{t-1}(\binom{t-1}{i}+\binom{t-1}{i-1})a_{r-t,s-i}.\]
 This is the same as 
 \[\binom{t}{0}a_{r-t,s}+\binom{t}{t}a_{r-t,s-t}+\sum_{i=1}^{t-1} \binom{t}{i} a_{r-t,s-i}\]
 or
 \[\sum_{i=1}^{t} \binom{t}{i} a_{r-t,s-i}\]
 and the lemma follows.
 \end{proof}
 
 \indent We will also need Lemma 8 from \cite{Paper3}, which says that for $r \geq 0$, $\displaystyle{\sum_{i=1}^s a_{r,s}=2^r}$.

\section{When $\mathbb{Z}_m^3$ is $H$-closed}
 \indent The last piece of tools with the $a_{r,s}$ coefficients we will be using is for the specific case where $n=3$. Here, since there are only three unique $a_{r,s}$ coefficients for each $r$, we will be denoting $a_r=a_{r,1}$, $b_r=a_{r,2}$, and $c_r=a_{r,3}$. Adapting some of our lemmas about the general $a_{r,s}$ coefficients, we have Corollary 7 from \cite{Paper5}, which says if $r>t>0$, then
 \begin{itemize}
 \item $a_{r+t}=a_ta_{r}+b_tc_{r}+c_tb_{r}$.
 
\item $b_{r+t}=a_tb_{r}+b_ta_{r}+c_tc_{r}$.

\item $c_{r+t}=a_tc_{r}+b_tb_{r}+c_ta_{r}$.
 \end{itemize}
 
 As for Lemma 8 from \cite{Paper3}, we have 
 \begin{corollary}\label{ars_sumto_n=3}
 Let $r\geq 0$. Then 
 \[a_r+b_r+c_r=2^r.\]
 \end{corollary}
 
 The last thing we need is Lemma 8 from \cite{Paper5}, which says for $n=3$,
 \begin{itemize}
    \item If $r \equiv 0 \; \text{mod} \; 6$, $a_r=b_r+1=c_r+1$.
    
    \item If $r \equiv 1 \; \text{mod} \; 6$, then $c_r=a_r-1=b_r-1$.
    
    \item If $r \equiv 2 \; \text{mod} \; 6$, then $b_r=a_r+1=c_r+1$.
    
    \item If $r \equiv 3 \; \text{mod} \; 6$, then $a_r=b_r-1=c_r-1$.
    
    \item If $r \equiv 4 \; \text{mod} \; 6$, then $c_r=a_r+1=b_r+1$.
    
    \item If $r \equiv 5 \; \text{mod} \; 6$, then $b_r=a_r-1=c_r-1$.
    
\end{itemize}
 
 Now we can prove Theorem \ref{HclosedTheorem}:
 \begin{proof}[Proof of Theorem \ref{HclosedTheorem}:]
 \textbf{(1):} Assume $m=2^l$. Our goal is to show $D^2(\mathbf{u})=H(\mathbf{u})$ for all $\mathbf{u} \in K(\mathbb{Z}_m^3)$. We begin by noting that in the proof of Proposition 5.2 of \cite{Dular}, they show that the basic Ducci sequence satisfies the condition that $D^2(\mathbf{u})=H(\mathbf{u})$ for $\mathbf{u}$ in the cycle of the basic Ducci cycle and $n=3$. We provide an alternate proof of this using the notation of $H$.
 
 \indent By Theorem 2 of \cite{Paper3}, $L_m(3)=l$. It suffices to show 
 \[D^{l+2}(0,0,...,0,1)=H(D^l(0,0,...,0,1)).\]
  We can prove this by showing 
    \[a_{l+2} \equiv c_l \; \text{mod} \; 2^l\]
    \[b_{l+2} \equiv a_l \; \text{mod} \; 2^l\]
    \[c_{l+2} \equiv b_l \; \text{mod} \; 2^l.\]
    Notice by Corollary 7 of \cite{Paper5}, we have 
    \[a_{l+2}=a_la_2+b_lc_2+c_lb_2.\]
    This is $a_l+b_l+2c_l$ or    
    \[(a_l+b_l+c_l)+c_l.\]
    By Corollary \ref{ars_sumto_n=3}, this is
    \[2^l+c_l,\]
    which is equivalent to $c_l \; \text{mod} \; m$.
   The other equivalencies $b_{l+2} \equiv a_l \; \text{mod} \; m$ and $c_{l+2} \equiv b_l \; \text{mod} \; m$ can be proved similarly and
    Part (1) follows from here.
    
    \textbf{(2):} Let $m$ be prime such that $m \equiv 5 \; \text{mod} \; 6$. We want to show $D^{m-1}(\mathbf{u})=H(\mathbf{u})$ for every $\mathbf{u} \in \mathbb{Z}_m^3$. Notice $m-1 \equiv 4 \; \text{mod} \; 6$ gives us $a_{m-1}=b_{m-1}=c_{m-1}-1$ by Lemma 8 of \cite{Paper5}. By Corollary \ref{ars_sumto_n=3},
    \[a_{m-1}+b_{m-1}+c_{m-1}=2^{m-1}.\]
    By Fermat's Little Theorem, a proof of which can be found in Theorem 5.1 of \cite{Burton}, $2^{m-1} \equiv 1 \; \text{mod} \; m$. Therefore, this yields
    \[3a_{m-1}+1 \equiv 1 \; \text{mod} \; m,\]
    which forces
    \[a_{m-1} \equiv 0 \; \text{mod} \; m.\]
    
    This also give us $b_{m-1} \equiv 0 \; \text{mod} \; m$ and $c_{m-1} \equiv 1 \; \text{mod} \; m$. Therefore, 
    \[D^{m-1}(x_1, x_2, x_3)=(x_3, x_1, x_2)\]
     and
    \[D^{m-1}(x_1, x_2, x_3)=H^2(x_1, x_2, x_3).\]

    \textbf{(3):} Let $m=2^lp$ where $p \equiv 5 \; \text{mod} \; m$ is prime. We want to show $D^{p-1}(\mathbf{u})=H(\mathbf{u})$ for every $\mathbf{u} \in K(\mathbb{Z}_m^3)$. By Theorem 2 of \cite{Paper3}, $L_{m}=l$. It suffices to show $D^{l+p-1}(0,0,1)=H^2(D^l(0,0,1))$ for every $\mathbf{u} \in \mathbb{Z}_m^3$. Therefore, we only need to show
    \[a_{l+p-1} \equiv b_l \; \text{mod} \; m\]
     \[b_{l+p-1} \equiv c_l \; \text{mod} \; m\]
      \[c_{l+p-1} \equiv a_l \; \text{mod} \; m.\]
      Notice
      \[a_{l+p-1}=a_la_{p-1}+b_lc_{p-1}+c_lb_{p-1}.\]
      By Lemma 8 of \cite{Paper5}, this is
      \[a_lb_{p-1}+b_l(b_{p-1}+1)+c_lb_{p-1}.\]
      Rearranging, this is
      \[(a_l+b_l+c_l)b_{p-1}+b_l,\]
      or by Corollary \ref{ars_sumto_n=3},
      \[2^lb_{p-1}+b_l.\]
      As we saw in the proof of (2), $b_{p-1} \equiv 0 \; \text{mod} \; p$. It then follows that 
      \[a_{l+p-1} \equiv b_l \; \text{mod} \; m.\] We can similarly prove that $b_{l+p-1} \equiv c_l \; \text{mod} \; m$ and $c_{l+p-1} \equiv a_l \; \text{mod} \; m$.
      Part (3) follows from here.
 \end{proof}
 
 \section{Other Cases where $\mathbb{Z}_m^n$ is $H$-closed}

\indent We now expand back out to other values of $n$ to examine which values of $n,m$ cause $\mathbb{Z}_m^n$ to be $H$-closed or weakly $H$-closed. Besides the $n=3$ discussed in the previous section, we began by checking the cases where $4 \leq n \leq 12$ and $3 \leq m \leq 12$. We then examined some additional cases for these values of $n$ to further test hypotheses that arose during the investigation, including for $n=16,32$ to further examine the case when $n$ is a power of $2$. Recall that when $n,m$ are both powers of $2$, $\mathbb{Z}_m^n$ is $H$-closed, but there is not a $\alpha>0$ and $-n < \beta < n$, $\beta \neq 0$ such that $D^{\alpha}(\mathbf{v})=H^{\beta}(\mathbf{v})$ for $\mathbf{v} \in K(\mathbb{Z}_m^n)$ that we are interested in, so they are not included in the following tables.
 
\indent  Figure \ref{Hclosed_cases} includes $n,m$ where $\mathbb{Z}_m^n$ is $H$-closed and the smallest value of $\alpha$ such that if $\mathbf{v} \in K(\mathbb{Z}_m^n)$, then $D^{\alpha}(\mathbf{v})=H^{\beta}(\mathbf{v})$, for some $-n < \beta < n$ and $\beta \neq 0$.

\begin{figure}
\centering

\begin{tabular}{|c|c|c|}
\hline
$n$ & $m$ & Let $\mathbf{v} \in K(\mathbb{Z}_m^n)$\\
\hline
4 & 3 & $D^2(\mathbf{v})=H^{-1}(\mathbf{v})$\\
  & 6 & $D^2(\mathbf{v})=H^{-1}(\mathbf{v})$\\
  & 7 & $D^6(\mathbf{v})=H^{-1}(\mathbf{v})$\\
  & 9 & $D^6(\mathbf{v})=H(\mathbf{v})$\\
  & 11 & $D^{10}(\mathbf{v})=H^{-1}(\mathbf{v})$\\
  & 12 & $D^2(\mathbf{v})=H^{-1}(\mathbf{v})$\\
  & 14 & $D^6(\mathbf{v})=H^{-1}(\mathbf{v})$\\
  & 19 & $D^{18}(\mathbf{v})=H^{-1}(\mathbf{v})$ \\
  & 21 & $D^{12}(\mathbf{v})=H^{-1}(\mathbf{v})$\\
  & 33 & $D^{10}(\mathbf{v})=H^{-1}(\mathbf{v})$\\
  \hline
5 & 3 & $D^8(\mathbf{v})=H^{-1}(\mathbf{v})$\\
  & 4 & $D^6(\mathbf{v})=H^{-2}(\mathbf{v})$\\
  & 6 & $D^{24}(\mathbf{v})=H^2(\mathbf{v})$\\
  & 7 & $D^{48}(\mathbf{v})=H^{-1}(\mathbf{v})$\\
  & 8 & $D^{12}(\mathbf{v})=H(\mathbf{v})$\\
  & 9 & $D^{24}(\mathbf{v})=H^2(\mathbf{v})$\\
  & 12 & $D^{24}(\mathbf{v})=H^2(\mathbf{v})$\\ 
  & 13 & $D^{84}(\mathbf{v})=H^{-2}(\mathbf{v})$\\
  & 17 & $D^{72}(\mathbf{v})=H(\mathbf{v})$\\
  & 19 & $D^{18}(\mathbf{v})=H^{-1}(\mathbf{v})$\\
  & 21 & $D^{48}(\mathbf{v})=H^{-1}(\mathbf{v})$\\
  & 23 & $D^{528}(\mathbf{v})=H^{-1}(\mathbf{v})$\\
  & 29 & $D^{28}(\mathbf{v})=H^{-1}(\mathbf{v})$\\
  \hline
6 & 5 & $D^4(\mathbf{v})=H^{-1}(\mathbf{v})$\\
  & 11 & $D^{10}(\mathbf{v})=H^{-1}(\mathbf{v})$\\
  & 17 & $D^{16}(\mathbf{v})=H^{-1}(\mathbf{v})$\\
  & 25 & $D^{20}(\mathbf{v})=H(\mathbf{v})$\\
  & 121 & $D^{110}(\mathbf{v})=H(\mathbf{v})$\\
  & 125 & $D^{100}(\mathbf{v})=H^{-1}(\mathbf{v})$\\
  \hline
7 & 3 & $D^{26}(\mathbf{v})=H^{-1}(\mathbf{v})$\\
  & 5 & $D^{124}(\mathbf{v})=H^{-1}(\mathbf{v})$\\
  & 9 & $D^{78}(\mathbf{v})=H^{-3}(\mathbf{v})$\\
  & 13 & $D^{12}(\mathbf{v})=H^{-1}(\mathbf{v})$\\
  & 15 & $D^{1612}(\mathbf{v})=H(\mathbf{v})$\\
  & 19 & $D^{2286}(\mathbf{v})=H^2(\mathbf{v})$\\
  & 27 & $D^{234}(\mathbf{v})=H^{-2}(\mathbf{v})$\\
  & 41 & $D^{40}(\mathbf{v})=H^{-1}(\mathbf{v})$\\
  \hline
8 & 7 & $D^6(\mathbf{v})=H^{-1}(\mathbf{v})$\\
  & 23 & $D^{22}(\mathbf{v})=H^{-1}(\mathbf{v})$\\
  & 31 & $D^{30}(\mathbf{v})=H^{-1}(\mathbf{v})$\\
  & 47 & $D^{46}(\mathbf{v})=H^{-1}(\mathbf{v})$\\  
  & 49 & $D^{42}(\mathbf{v})=H(\mathbf{v})$\\
  & 98 & $D^{42}(\mathbf{v})=H(\mathbf{v})$\\
  & 161 & $D^{66}(\mathbf{v})=H^{-3}(\mathbf{v})$\\
  & 322 & $D^{66}(\mathbf{v})=H^{-3}(\mathbf{v})$\\
  & 2254 & $D^{462}(\mathbf{v})=H^3(\mathbf{v})$\\
  & 3937 & $D^{630}(\mathbf{v})=H^3(\mathbf{v})$\\
  \hline
  \end{tabular}
  \quad
  \begin{tabular}{|c|c|c|}
\hline
$n$ & $m$ & Let $\mathbf{v} \in K(\mathbb{Z}_m^n)$\\
\hline
9 & 2 & $D^7(\mathbf{v})=H^{-1}(\mathbf{v})$\\
  & 4 & $D^{14}(\mathbf{v})=H^{-2}(\mathbf{v})$\\
  & 5 & $D^{124}(\mathbf{v})=H^{-1}(\mathbf{v})$\\
  & 8 & $D^{28}(\mathbf{v})=H^{-4}(\mathbf{v})$\\
  & 10 & $D^{868}(\mathbf{v})=H^2(\mathbf{v})$\\
  & 11 & $D^{1330}(\mathbf{v})=H^{-1}(\mathbf{v})$\\
  & 17 & $D^{16}(\mathbf{v})=H^{-1}(\mathbf{v})$\\
  & 53 & $D^{53}(\mathbf{v})=H^{-1}(\mathbf{v})$\\
  \hline
10 & 3 & $D^8(\mathbf{v})=H^{-1}(\mathbf{v})$\\
   & 7 & $D^{48}(\mathbf{v})=H^{-1}(\mathbf{v})$\\
   & 9 & $D^{24}(\mathbf{v})=H^{-3}(\mathbf{v})$\\  
   & 13 & $D^{168}(\mathbf{v})=H^{-1}(\mathbf{v})$\\
   & 17 & $D^{288}(\mathbf{v})=H^{-1}(\mathbf{v})$\\
   & 19 & $D^{18}(\mathbf{v})=H^{-1}(\mathbf{v})$\\
   & 29 & $D^{28}(\mathbf{v})=H^{-1}(\mathbf{v})$\\
   & 361 & $D^{342}(\mathbf{v})=H(\mathbf{v})$\\
   & 841 & $D^{812}(\mathbf{v})=H(\mathbf{v})$\\
 \hline
11 & 2 & $D^{31}(\mathbf{v})=H^{-1}(\mathbf{v})$\\
   & 4 & $D^{62}(\mathbf{v})=H^{-2}(\mathbf{v})$\\
   & 7 & $D^{16806}(\mathbf{v})=H^{-1}(\mathbf{v})$\\
   & 8 & $D^{124}(\mathbf{v})=H^{-4}(\mathbf{v})$\\
   & 13 & $D^{371292}(\mathbf{v})=H^{-1}(\mathbf{v})$\\
   & 16 & $D^{248}(\mathbf{v})=H^3(\mathbf{v})$\\
   & 17 & $D^{709928}(\mathbf{v})=H^5(\mathbf{v})$\\
   & 43 & $D^{42}(\mathbf{v})=H^{-1}(\mathbf{V})$\\
   & 109 & $D^{108}(\mathbf{v})=H^{-1}(\mathbf{v})$\\
   \hline
12 & 11 & $D^{10}(\mathbf{v})=H^{-1}(\mathbf{v})$\\
   & 23 & $D^{22}(\mathbf{v})=H^{-1}(\mathbf{v})$\\
   & 121 & $D^{110}(\mathbf{v})=H(\mathbf{v})$\\
   & 529 & $D^{506}(\mathbf{v})=H(\mathbf{v})$\\
\hline
16 & 31 & $D^{30}(\mathbf{v})=H^{-1}(\mathbf{v})$\\
   & 47 & $D^{46}(\mathbf{v})=H^{-1}(\mathbf{v})$\\
   & 62 & $D^{30}(\mathbf{v})=H^{-1}(\mathbf{v})$\\
   & 79 & $D^{78}(\mathbf{v})=H^{-1}(\mathbf{v})$\\
   & 124 & $D^{30}(\mathbf{v})=H^{-1}(\mathbf{v})$\\  
   & 961 & $D^{930}(\mathbf{v})=H(\mathbf{v})$\\
   & 1922 & $D^{930}(\mathbf{v})=H(\mathbf{v})$\\
   & 3937 & $D^{630}(\mathbf{v})=H^{-5}(\mathbf{v})$\\
   & 751967 & $D^{11970}(\mathbf{v})=H(\mathbf{v})$\\
  \hline
32 & 31 & $D^{30}(\mathbf{v})=H^{-1}(\mathbf{v})$\\
   
   & 62 & $D^{30}(\mathbf{v})=H^{-1}(\mathbf{v})$\\
   & 124 & $D^{30}(\mathbf{v})=H^{-1}(\mathbf{v})$\\
   & 127 & $D^{126}(\mathbf{v})=H^{-1}(\mathbf{v})$\\
   & 191 & $D^{190}(\mathbf{v})=H^{-1}(\mathbf{v})$\\
   & 961 & $D^{930}(\mathbf{v})=H(\mathbf{v})$\\
   & 1922 & $D^{930}(\mathbf{v})=H(\mathbf{v})$\\
   & & \\
   \hline

\end{tabular}
\caption{Cases of $n,m$ where $\mathbb{Z}_m^n$ is $H$-closed}
\label{Hclosed_cases}
\end{figure}

Figure \ref{weaklyHclosed_cases} provides the cases where $\mathbb{Z}_m^n$ is weakly $H$-closed, but not $H$-closed, along with the smallest $\alpha>0$ for some $-n < \beta <n$, $\beta \neq -1, 0, 1$ such that $D^{\alpha}(\mathbf{v})=H^{\beta}(\mathbf{v})$ when $\mathbf{v} \in K(\mathbb{Z}_m^n)$.

\begin{figure}
\centering

\begin{tabular}{|c|c|c|}
\hline
$n$ & $m$ & Let $\mathbf{v} \in K(\mathbb{Z}_m^n)$\\
\hline
4 & 21 & $D^{12}(\mathbf{v})=H^2(\mathbf{v})$\\
  & 77 & $D^{60}(\mathbf{v})=H^2(\mathbf{v})$\\\
  & 539 & $D^{420}(\mathbf{v})=H^2(\mathbf{v})$\\
  & 847 & $D^{660}(\mathbf{v})=H^2(\mathbf{v})$\\
\hline
6 & 4 & $D^4(\mathbf{v})=H^2(\mathbf{v})$\\
  & 8 & $D^8(\mathbf{v})=H^{-2}(\mathbf{v})$\\
  & 10 & $D^8(\mathbf{v})=H^{-2}(\mathbf{v})$\\
  & 22 & $D^{20}(\mathbf{v})=H^{-2}(\mathbf{v})$\\
  & 34 & $D^{32}(\mathbf{v})=H^{-2}(\mathbf{v})$\\
  & 55 & $D^{40}(\mathbf{v})=H^2(\mathbf{v})$\\
\hline
8 & 4991 & $D^{660}(\mathbf{v})=H^2(\mathbf{v})$\\
  & 9982 & $D^{660}(\mathbf{v})=H^2(\mathbf{v})$\\
  & 234577 & $D^{15180}(\mathbf{v})=H^{-2}(\mathbf{v})$\\
  \hline
9 & 7 & $D^{114}(\mathbf{v})=H^3(\mathbf{v})$\\
  & 13 & $D^{732}(\mathbf{v})=H^{-3}(\mathbf{v})$\\
  & 14 & $D^{798}(\mathbf{v})=H^3(\mathbf{v})$\\
  & 35 & $D^{7068}(\mathbf{v})=H^{-3}(\mathbf{v})$\\
  \hline
 
10 & 2 & $D^6(\mathbf{v})=H^{-2}(\mathbf{v})$\\
   & 4 & $D^{12}(\mathbf{v})=H^{-4}(\mathbf{v})$\\
   & 6 & $D^{48}(\mathbf{v})=H^4(\mathbf{v})$\\
   & 8 & $D^{24}(\mathbf{v})=H^2(\mathbf{v})$\\
   & 12 & $D^{48}(\mathbf{v})=H^4(\mathbf{v})$\\
   & 21 & $D^{96}(\mathbf{v})=H^{-2}(\mathbf{v})$\\
   & 26 & $D^{336}(\mathbf{v})=H^{-2}(\mathbf{v})$\\
   
\hline
12 & 3 & $D^6(\mathbf{v})=H^{-3}(\mathbf{v})$\\
   & 4 & $D^8(\mathbf{v})=H^4(\mathbf{v})$\\
   & 8 & $D^{16}(\mathbf{v})=H^{-4}(\mathbf{v})$\\
   & 9 & $D^{18}(\mathbf{v})=H^3(\mathbf{v})$\\
   & 16 & $D^{32}(\mathbf{v})=H^4(\mathbf{v})$\\
   & 27 & $D^{54}(\mathbf{v})=H^{-3}(\mathbf{v})$\\
   & 33 & $D^{30}(\mathbf{v})=H^{-3}(\mathbf{v})$\\
 \hline
16 & 1457 & $D^{1380}(\mathbf{v})=H^2(\mathbf{v})$\\
   & 2914 & $D^{1380}(\mathbf{v})=H^2(\mathbf{v})$\\   
   & 45167 & $D^{42780}(\mathbf{v})=H^{-2}(\mathbf{v})$\\
  \hline
32 & 3937 & $D^{1260}(\mathbf{v})=H^{-10}(\mathbf{v})$\\
   & 751967 & $D^{23940}(\mathbf{v})=H^2(\mathbf{v})$\\
   \hline

\end{tabular}
\caption{Cases of $n,m$ where $\mathbb{Z}_m^n$ is Weakly $H$-closed}
\label{weaklyHclosed_cases}
\end{figure}

Figure \ref{notHclosed_cases} provides the remaining cases we tested where we could conclude that $\mathbb{Z}_m^n$ was not $H$-closed or weakly $H$-closed.

\begin{figure}
\centering
\begin{tabular}{|c|c|}
\hline
n & m\\
\hline
4 & 5\\
  & 10\\
  & 13\\
  \hline
5 & 5\\
  & 10\\
  & 11\\
  \hline
6 & 3\\
  & 6\\
  & 7\\
  & 9\\
  & 12\\
  \hline
  
7 & 2\\
  & 4\\
  & 6\\
  & 7\\
  & 8\\
  & 10\\
  & 11\\
  & 12\\
  & \\
   &  \\
   & \\
   &  \\
   & \\
   &  \\
   & \\
  \hline
  \end{tabular}
 \quad
 \quad
 \quad
 \quad
  \begin{tabular}{|c|c|}
\hline
n & m\\
\hline
8 & 3\\
  & 5\\
  & 6\\
  & 9\\
  & 10\\
  & 11\\
  & 12\\
  & 21\\
  & 63\\
  \hline
 
9 & 3\\
  & 6\\
  & 9\\
  & 12\\
  & 15\\
  \hline
10 & 5\\
   & 10\\
   & 11\\
   \hline
11 & 3\\
   & 5\\
   & 6\\
   & 9\\
   & 10\\
   & 11\\
   & 12\\
   & \\
   &  \\
   \hline
   
    \end{tabular}
  \quad
 \quad
 \quad
 \quad
  \begin{tabular}{|c|c|}
\hline
n & m\\
\hline
12 & 2\\
   & 5\\
   & 6\\
   & 7\\
   & 10\\
   & 12\\
 \hline
 
16 & 3\\
   & 5\\
   & 6\\
   & 7\\
   & 9\\
   & 10\\
   & 11\\
   & 12\\
   & 93\\
   & 279\\
   & 85399\\ 
   \hline
32 & 3\\
   & 5\\
   & 6\\
   & 7\\
   & 9\\
   & 10\\
   & 12\\  
   & 93\\
   & 279\\
   \hline

\end{tabular}
\caption{Cases of $n,m$ where $\mathbb{Z}_m^n$ is not Weakly $H$-closed}
\label{notHclosed_cases}
\end{figure}

\indent To determine if $\mathbb{Z}_m^n$ was $H$-closed, weakly $H$-closed, or neither, we began by using MATLAB to calculate $L_m(n)$ and $P_m(n)$. Recall that we can focus on whether the basic Ducci sequence is $H$-closed to determine if $\mathbb{Z}_m^n$ is $H$-closed. If $L=L_m(n)$ and there existed $\alpha>0$ and $-n <  \beta < n$, $\beta \neq 0$ such that 
\[D^{L+\alpha}(0,0,...,,0,1)=H^{L+\beta}(0,0,...,0,1),\]
 then there would need to be $\gamma>0$ such that $\alpha*\gamma=P_m(n)$. If $\mathbb{Z}_m^n$ was $H$-closed, then $\gamma=n$. If $\mathbb{Z}_m^n$ was weakly $H$-closed, but not $H$-closed, then $\gamma<n$ and $gcf(\alpha, n) \neq 1$. We used this knowledge to determine if $\mathbb{Z}_m^n$ was $H$-closed, weakly $H$-closed, or neither. 

\indent From Figures \ref{Hclosed_cases}, \ref{weaklyHclosed_cases}, and \ref{notHclosed_cases}, we can begin to hypothesize what values of $n,m$ make $\mathbb{Z}_m^n$ $H$-closed or weakly $H$-closed, as well as what $\alpha>0$ and $-n < \beta <n$, $\beta \neq 0$ satisfy $D^{\alpha}(\mathbf{v})=H^{\beta}(\mathbf{v})$ when $\mathbf{v} \in K(\mathbb{Z}_m^n)$. The first is that if $n$ is even, $m \equiv -1 \; \text{mod} \; n$ is prime, and $\mathbf{v} \in K(\mathbb{Z}_m^n)$, then $D^{m-1}(\mathbf{v})=H^{-1}(\mathbf{v})$, which is Theorem \ref{Hclosed_even} and which we will prove later in this section. 

\indent We believe we can extend this further: if $n$ is even, $m=p^l$ and $p \equiv -1 \; \text{mod} \; n$ where $p$ is prime and $l \geq 1$, then 
\[D^{\phi(m)}(\mathbf{v})=H^{(-1)^l}(\mathbf{v})\]
when $\mathbf{v} \in K(\mathbb{Z}_m^n)$. If true, this would result in $\mathbb{Z}_m^n$ being $H$-closed for this case. 

\indent For the specific case where $n$ is a power of $2$, it also appears that if $m=2^{l_2}p^l$ where $p \equiv -1 \; \text{mod} \; n$ is prime and $l_2 \geq 0$, then 
\[D^{p-1}(\mathbf{v})=H^{(-1)^{l_2}}(\mathbf{v})\]
when $\mathbf{v} \in K(\mathbb{Z}_m^n)$, providing another case when $\mathbb{Z}_m^n$ is $H$-closed, if true. 

\indent In the cases discussed in the previous paragraphs, $\alpha$ is chosen to be the smallest value such that $D^{\alpha}(\mathbf{v})=H^{\beta}(\mathbf{v})$ for $\mathbf{v} \in K(\mathbb{Z}_m^n)$ and $-n < \beta < n$, $\beta \neq 0$. However, it also appears that when $n$ is a power of $2$ and $m=p_1^{l_1}p_2^{l_2}$, $p_1, p_2 \equiv -1 \; \text{mod} \; n$ are prime and $l_1, l_2 \geq 1$ results in $\mathbb{Z}_m^n$ being $H$-closed or weakly $H$-closed. As for the smallest $\alpha$ such that $D^{\alpha}(\mathbf{v})=H^{\beta}(\mathbf{v})$ when $\mathbf{v} \in K(\mathbb{Z}_m^n)$ and $-n < \beta < n$, $\beta \neq 0$, a standard formula for how $\alpha$ is related to $m$ is not as clear. The exception here is when $n=4$ and $p_1, p_2 \neq 3$. Here, it appears that $D^{\phi(m)}(\mathbf{v})=H^2(\mathbf{v})$ when $\mathbf{v} \in K(\mathbb{Z}_m^n)$. If this is true, then $\mathbb{Z}_m^4$ would be weakly $H$-closed.

\indent These are all of the cases that we have been able to identify up to this point and that we would currently like to try proving at some point. For now, we will focus on proving the one in Theorem \ref{Hclosed_even}:
 
 \begin{proof}[Proof of Theorem \ref{Hclosed_even}]
 \indent Assume $n$ is even and $m \equiv -1 \; \text{mod} \; n$ where $m$ is prime. By Theorem 2 of \cite{Paper4}, $L_m(n)=1$ because $n$ is even and $n$ and $m$ are relatively prime. Therefore, it suffices to show 
 \[D^m(0,0,...,0,1)=H^{-1}(D(0,0,...,1)),\]
 or that 
 \[D^m(0,0,...,0,1)=(1,0,0,....,0,1).\]
 This will be the case if 
 \[a_{m,s} \equiv 
 \begin{cases}
 0 \; \text{mod} \; m & 1<s<n\\
  1 \; \text{mod} \; m & s=1,n
 \end{cases}.\]
 
 \indent Write $m=\delta n -1$. Using Lemma \ref{2nd_ars_sum}, we can break down $a_{\delta n-1,s}$ as follows:
 
 \[a_{\delta n -1, s} =\sum_{i=0}^{\delta n -1} \binom{\delta n -1}{i} a_{0,s-i}.\]
 For the term of this sum to be nonzero, we need that both $\displaystyle{\binom{\delta n -1}{i}} \not \equiv 0 \; \text{mod} \; m$ and $a_{0,s-i} \neq 0$. The only time $a_{0, s-i} \neq 0$ is when $s-i \equiv 1 \; \text{mod} \; n$, in which case $a_{0,1}=1$. The two cases where $\displaystyle{\binom{\delta n -1}{i} } \not \equiv 0 \; \text{mod} \; m$ are when $i=0, \delta n-1$, in which case $\displaystyle{\binom{\delta n -1}{i}=1}$. 
 
 \indent If $i=0$, then $a_{0, s-i}=0$ unless $s=1$. Which means $a_{\delta n -1, 1}=1$. If $i=\delta n-1$, then $a_{0, s-i}=0$ unless $s=n$. Which means $a_{\delta n -1, n}=1$. For $s \neq 0, n$, all of the terms of the sum are $0$, so $a_{\delta n-1, s}=0$. From here, the theorem follows.
 \end{proof}


\begin{thebibliography}{99}
\bibitem {Breuer1} Breuer, F. (1999). Ducci Sequences Over Abelian Groups. \textit{Communications in Algebra, 27(12)}, 5999-6013.

\bibitem {Breuer2} Breuer, F. (2010). Ducci Sequences and Cyclotomic Fields. \textit{Journal of Difference Equations and Applications, 16(7)}, 847-862.

\bibitem{Brown} Brown, R. \& Merzel, J. (2007). The Length of Ducci's Four Number Game \textit{Rocky Mountain Journal of Mathematics, 37(1)}, 45-65.

\bibitem{Burmester} Burmester, M., Forcade, R., \& Jacobs, E. (1978) Circles of Numbers. \textit{Glasgow Mathematical Journal, 19}, 115-119.

\bibitem{Burton} Burton, D. (2011). \textit{Elementary Number Theory} (7th ed.). New York, NY: McGraw Hill.

\bibitem{Chamberland} Chamberland, M. (2003). Unbounded Ducci Sequences. \textit{Journal of Difference Equations and Applications, 9(10)}, 887-895.

\bibitem{Dular} Dular, B. (2020). Cycles of Sums of Integers. \textit{Fibonacci Quarterly, 58(2)}, 126-139.

 \bibitem{Ehrlich} Ehrlich, A. (1990). Periods in Ducci's $n$-Number Game of Differences. \textit{Fibonacci Quarterly, 28(4)}, 302-305. 
 
  \bibitem{Freedman} Freedman, B (1948). The Four Number Game. \textit{Scripta Mathematica, 14}, 35-47.
  
 \bibitem {Glaser} Glaser, H. \& Sch\"{o}ffl, G. (1995). Ducci Sequences and Pascal's Triangle. \textit{Fibonacci Quarterly, 33(4)}, 313-324.
 
 \bibitem{Paper1} Lewis, M.L. \& Tefft, S.M. (2024). The Period of Ducci Cycles on $\mathbb{Z}_{2^l}$ for Tuples of Length $2^k$. Submitted for Publication. $<$arXiv: 2401.17502$>$ 
 
 \bibitem{Paper3} Lewis, M.L. \& Tefft, S.M. (2024). Ducci on $\mathbb{Z}_m^n$ and the Maximum Length for $n$ Odd. Submitted for Publication. 
$<$arXiv: 2401.05319$>$


\bibitem{Paper4} Lewis, M.L. \& Tefft, S.M. (2024). The Maximum Length for Ducci Sequences on $\mathbb{Z}_m^n$ when $n$ is Even. $<$arXiv: 2410.18204$>$

\bibitem{Paper5} Lewis, M.L. \& Tefft, S.M. (2024). Values of Ducci Periods for Sequences on $\mathbb{Z}_m^n$. 
 
 \bibitem{Furno}
       Ludington Furno, A (1981). Cycles of differences of integers. \textit{Journal of Number Theory, 13(2)}, 255-261. 
       
 \bibitem{MATLAB} The MathWorks Inc. (2023).  MATLAB version 9.14.0 (R2023a), Natick, Massachusetts: The MathWorks Inc. https.//www.mathworks.com. 
       
       \bibitem {Schinzel} Misiurewicz, M., \& Schinzel, A. (1988). On $n$ Numbers in a Circle. \textit{Hardy Ramanujan Journal, 11}, 30-39.
       
        \bibitem{Wong} Wong, F.B. (1982). Ducci Processes. \textit{The Fibonacci Quarterly, 20(2)}, 97-105.
\end{thebibliography}
\end{document}